\documentclass{amsart}

\usepackage[mathscr]{eucal}
\usepackage{amssymb}
\usepackage[usenames,dvipsnames]{color}
\usepackage[normalem]{ulem}
\usepackage{amsthm}
\usepackage{bbold}
\usepackage{enumerate}
\usepackage{array}
\usepackage{amsmath}
\usepackage{etoolbox}

\usepackage{hyperref}

\hypersetup{colorlinks=true,citecolor=BrickRed,urlcolor=BrickRed,linkcolor=BrickRed}

\numberwithin{equation}{section}
\usepackage[all]{xy}

\newdir{ >}{{}*!/-10pt/\dir{>}}

\swapnumbers 

\newtheorem{Thm}[equation]{Theorem}
\newtheorem*{Thm*}{Theorem}
\newtheorem{Cor}[equation]{Corollary}
\theoremstyle{remark}

\newtheorem{Ter}[equation]{Terminology}
\newtheorem{Exa}[equation]{Example}

\newtheorem{Rem}[equation]{Remark}

\newcommand{\nc}{\newcommand}
\nc{\dmo}{\DeclareMathOperator}

\dmo{\EM}{\textit{H}}  
\dmo{\Ab}{Ab}
\dmo{\Abelem}{Abelem}
\dmo{\Add}{Add}
\dmo{\Aut}{Aut}
\dmo{\Bi}{bi}
\dmo{\Bisets}{Bisets}
\dmo{\CAT}{CAT}
\dmo{\coev}{coev}
\dmo{\Coloc}{Coloc}
\dmo{\ev}{ev}
\dmo{\Fib}{Fib}
\dmo{\Free}{Free}
\dmo{\Id}{Id}
\dmo{\Loc}{Loc}
\dmo{\rmI}{I}
\dmo{\rmL}{L}
\dmo{\rmR}{R}
\dmo{\Spc}{Spc}
\dmo{\Thick}{Thick}
\dmo{\chara}{char}%
\dmo{\coh}{coh} 
\dmo{\Coind}{CoInd}
\dmo{\coker}{coker}

\dmo{\cone}{cone}
\dmo{\Der}{D}
\nc{\Rder}{\mathrm{R}} 
\nc{\Lder}{\mathrm{L}} 
\dmo{\Khocat}{K}
\dmo{\End}{End}
\dmo{\Ext}{Ext}
\dmo{\rmH}{H}
\dmo{\Ho}{Ho}
\dmo{\Hom}{Hom}
\dmo{\id}{id}
\dmo{\Img}{Im}
\dmo{\incl}{incl}
\dmo{\Ind}{Ind}
\dmo{\CoInd}{CoInd}
\dmo{\Ker}{Ker}
\dmo{\Les}{Les}
\dmo{\Map}{Map}%
\dmo{\Mod}{Mod}
\dmo{\GrMod}{GrMod}
\dmo{\lax}{lax}
\dmo{\modname}{mod}%
\dmo{\grmod}{grmod}
\dmo{\Mor}{Mor}%
\dmo{\Obj}{Obj}
\dmo{\opname}{op}
\dmo{\Or}{Or}
\dmo{\pr}{pr}
\dmo{\canin}{in} 
\dmo{\Proj}{Proj} 
\dmo{\proj}{proj}
\dmo{\Qcoh}{Qcoh}
\dmo{\rank}{rank}
\dmo{\Res}{Res}
\dmo{\Rname}{R}
\dmo{\SH}{SH}
\nc{\SHp}{\SH_{(p)}}
\dmo{\smallb}{b}
\dmo{\smallperf}{perf}
\dmo{\Span}{Span}
\dmo{\Spec}{Spec}
\dmo{\Stab}{Stab}
\dmo{\stab}{stab}
\dmo{\supp}{supp}
\dmo{\switch}{switch}
\dmo{\TTR}{TTR}

\nc{\Beren}[1]{{\color{MidnightBlue}#1}}
\nc{\Ivo}[1]{{\color{OliveGreen}#1}}
\nc{\Paul}[1]{{\color{Violet}#1}}
\nc{\Pout}[1]{\Paul{\sout{#1}}}
\nc{\Bout}[1]{\Beren{\sout{#1}}}
\nc{\Iout}[1]{\Ivo{\sout{#1}}}
\nc{\IFF}{$\Leftrightarrow$}
\nc{\DO}{\omega_f}
\nc{\Crelcpt}{\cat{C}^{c/f}}
\nc{\Crich}{\underline{\cat{C}}}
\nc{\DbG}{\Db(\kk G\mmod)}
\nc{\uA}{\underline{A}}
\nc{\doublequot}[3]{#1\backslash #2/#3}
\nc{\HGK}{\doublequot HGK}
\nc{\quadtext}[1]{\quad\textrm{#1}\quad}
\nc{\qquadtext}[1]{\qquad\textrm{#1}\qquad}
\nc{\PZG}{\cat C_{\bbZ}(\bbZ G)}
\nc{\TTRK}{\TTR(\cat K)}
\nc{\psets}{\mathsf{-sets}_\sbull}
\nc{\Gsets}{G\mathsf{-sets}}
\nc{\Hsets}{H\mathsf{-sets}}
\nc{\AddK}{\Add^{\Sigma}(\cat K)}
\nc{\adj}{\dashv}
\nc{\adjto}{\rightleftarrows}
\nc{\AK}{A\MModcat{K}}
\nc{\BK}{B\MModcat{K}}
\nc{\bbA}{\mathbb{A}}
\nc{\bbB}{\mathbb{B}}
\nc{\bbC}{\mathbb{C}}
\nc{\bbI}{\mathbb{I}}
\nc{\bbN}{\mathbb{N}}
\nc{\bbP}{\mathbb{P}}
\nc{\bbQ}{\mathbb{Q}}
\nc{\bbR}{\mathbb{R}}
\nc{\bbZ}{\mathbb{Z}}
\nc{\bbZp}{\mathbb{Z}_{(p)}}
\nc{\Sphere}{\mathbb{S}} 
\nc{\cat}[1]{\mathscr{#1}}
\nc{\Displ}{\displaystyle}
\nc{\ie}{{\sl i.e.}\ }
\nc{\eg}{{\sl e.g.}\ }
\nc{\into}{\mathop{\rightarrowtail}}
\nc{\inv}{^{-1}}
\nc{\isoto}{\buildrel \sim\over\to}
\nc{\isotoo}{\mathop{\buildrel \sim\over\too}}
\nc{\kk}{\Bbbk}
\nc{\onto}{\mathop{\twoheadrightarrow}}
\nc{\too}{\mathop{\longrightarrow}\limits}
\nc{\xytriangle}[7]{\xymatrix@C=#7em{#1\ar[r]^-{\Displ #4} & #2 \ar[r]^-{\Displ #5}&#3\ar[r]^-{\Displ #6}&T #1}}
\nc{\ababs}{{\sl ab absurdo}}
\nc{\adh}[1]{\overline{#1}}
\nc{\adhpt}[1]{\adh{\{#1\}}}
\nc{\aka}{{a.\,k.\,a.}\ }
\nc{\ala}{{\sl \`a la}\ }
\nc{\Autcat}[1]{\Aut_{\cat #1}}
\nc{\cO}{\mathcal{O}}
\nc{\calO}{\mathcal{O}}
\nc{\cV}{\mathcal{V}}
\nc{\Db}{\Der^{\smallb}}
\nc{\Dqc}{\Der_{\Qcoh}}
\nc{\Dperf}{\Der^{\smallperf}}
\nc{\eps}{\epsilon}
\nc{\FFree}{\,\text{-}\Free}%
\nc{\FFreecat}[1]{\FFree_{\cat #1}}
\nc{\FK}{\mathcal{F}(\cat K)}
\nc{\gm}{\mathfrak{m}}
\nc{\Homcat}[1]{\Hom_{\cat #1}}
\nc{\Morcat}[1]{\Mor_{\cat #1}}
\nc{\hook}{\hookrightarrow}
\nc{\Idcat}[1]{\Id_{\cat{#1}}}
\nc{\ideal}[1]{\langle #1\rangle}
\nc{\ihom}{{\mathsf{hom}}} 
\nc{\ihomcat}[1]{\ihom_{\cat #1}}
\nc{\Kcat}[1]{#1\MModcat{K}}
\nc{\KP}{\cat{K}_{\cat P}}
\nc{\loccit}{{\sl loc.\ cit.}}
\nc{\lind}{\rmL\!}
\nc{\RR}{\rmR\!}
\nc{\Lotimes}{\otimes^{\rmL}}
\nc{\Mid}{\,\big|\,}
\nc{\MMod}{\,\text{-}\Mod}%
\nc{\MModcat}[1]{\MMod_{\cat #1}}%
\nc{\mmod}{\,\text{--}\modname}%
\nc{\mmodb}{\mmod^\sbull}%
\nc{\op}{^{\opname}}
\nc{\oto}[1]{\overset{#1}\to}
\nc{\otoo}[1]{\overset{#1}{\,\too\,}}
\nc{\ourfrac}[2]{\genfrac{}{}{0pt}{}{\Displ #1}{\scriptstyle #2}}
\nc{\ouriff}{\Leftrightarrow}
\nc{\oursetminus}{\!\smallsetminus\!}
\nc{\potimes}[1]{^{\otimes #1}}
\nc{\pproj}{\,\text{-}\proj}
\nc{\ptimes}[1]{^{\times #1}}
\nc{\dd}[1]{_{{\scriptscriptstyle(#1)}}}
\nc{\uu}[1]{^{{\scriptscriptstyle(#1)}}}
\nc{\pushout}{\textrm{\rm p.o.}}
\nc{\qp}{q_{_{\scriptstyle \cat P}}\!}%
\nc{\Rcat}[1]{\Rname_{\cat #1}^\sbull}
\nc{\rdto}{}
\nc{\restr}[1]{_{|_{\scriptstyle #1}}}
\nc{\RK}{\Rcat{K}}
\nc{\sbull}{{\scriptscriptstyle\bullet}}
\nc{\SET}[2]{\big\{\,#1\Mid#2\,\big\}}
\nc{\SHA}{\SH{}^{\bbA^{1}}}
\nc{\SHfin}{\SH^{\text{\rm fin}}}
\nc{\smallmatrice}[1]{\left(\begin{smallmatrix} #1 \end{smallmatrix}\right)}
\nc{\SpcAK}{\Spc(A\MModcat{K})}
\nc{\SpcK}{\Spc(\cat K)}
\nc{\suppcat}[1]{\supp(\cat #1)}
\nc{\then}{\Rightarrow}
\nc{\tideal}[1]{\ideal{#1}}
\nc{\unit}{\mathbb{1}}
\nc{\unitcat}[1]{\unit_{\cat #1}}
\nc{\onept}{\mathrm{B}} 

\nc{\HG}{\!{}^{^H}\overline{G}}
\nc{\uY}{\widetilde{Y}}

\nc{\Dk}{\dual_{\kappa}}
\nc{\Dkk}{\dual_{\kappa'}}
\nc{\bs}{\backslash}
\nc{\biCpt}{\mathrm{biCpt}}
\nc{\biLCpt}{\mathrm{biLCpt}} 
\nc{\Grps}{\mathsf{Grps}}
\nc{\Sets}{\mathsf{Sets}}
\nc{\Top}{\mathsf{Top}}
\nc{\Comp}{\mathsf{Top}^{\mathsf{comp}}}

\nc{\lG}{{}_{{\color{Gray}\scriptscriptstyle G}}}
\nc{\lH}{{}_{{\color{Gray}\scriptscriptstyle H}}}
\nc{\rG}{_{{\color{Gray}\!\scriptscriptstyle G}}}
\nc{\rH}{_{{\color{Gray}\!\scriptscriptstyle H}}}
\nc{\rK}{_{{\color{Gray}\!\scriptscriptstyle K}}}

\nc{\dual}{\Delta}

\nc{\ra}{\rightarrow}
\nc{\xra}{\xrightarrow}

\nc{\C}{\mathbb{C}} 
\nc{\Cont}{\mathrm{C}} 
\nc{\KK}{\mathsf{KK}}
\nc{\Modules}{\mathsf{Mod}}
\nc{\Alg}{\mathsf{Alg}}
\nc{\Sep}{\mathsf{Sep}}
\begin{document}
\def\Rdef{1.7em}


\title{A note on triangulated monads and categories of module spectra}
\author{Ivo Dell'Ambrogio}
\author{Beren Sanders}
\date{\today}

\address{Ivo Dell'Ambrogio, Laboratoire de Math\'ematiques Paul Painlev\'e, Universit\'e de \break Lille~1, Cit\'e Scientifique -- B\^at.~M2, 59665 Villeneuve-d'Ascq Cedex, France}
\email{ivo.dellambrogio@math.univ-lille1.fr}
\urladdr{http://math.univ-lille1.fr/$\sim$dellambr}

\address{Beren Sanders, Laboratory for Topology and Neuroscience, \'{E}cole Polytechnique F\'{e}d\'{e}rale de Lausanne, B\^at.~MA, Station 8, 1015 Lausanne, Switzerland}
\email{beren.sanders@epfl.ch}

\begin{abstract}
Consider a monad on an idempotent complete triangulated category with the property that its Eilenberg--Moore category of modules inherits a triangulation. We show that any other triangulated adjunction realizing  this monad is `essentially monadic', i.e.\ becomes monadic after performing the two evident necessary operations of  taking the Verdier quotient by the kernel of the right adjoint and idempotent completion. 
In this sense, the monad itself is `intrinsically monadic'.
It follows that for any highly structured ring spectrum, its category of homotopy (a.k.a.\ na\"ive) modules is triangulated if and only if it is equivalent to its category of highly structured (a.k.a.\ strict) modules. 
 \end{abstract}


\thanks{First-named author partially supported by the Labex CEMPI (ANR-11-LABX-0007-01)}
\thanks{Second-named author partially supported by the Danish National Research Foundation through the Centre for Symmetry and Deformation (DNRF92)}

\maketitle

\vskip-\baselineskip\vskip-\baselineskip

\section{Triangulated monads and their realizations}

Let $\cat C$ be an idempotent complete triangulated category.  Given a monad $\bbA$ on~$\cat C$, we can ask whether the Eilenberg--Moore category $\bbA\MMod_\cat C$ of $\bbA$-modules in $\cat C$ (a.k.a.\ $\bbA$-algebras) 
inherits the structure of a triangulated category.
Although this seems to rarely occur in Nature, it does happen in some notable situations,
for example if $\bbA$ is an idempotent monad (\ie a Bousfield localization).
More generally, Balmer \cite{Balmer11} proved that this is the case when the monad $\bbA$ is separable
(provided that $\cat C$ is endowed with an $\infty$-triangulation, which is always the case when it admits an underlying model or derivator).
Some non-separable examples are also known (see \cite{Gutierrez05} and Example~\ref{exa:EML} below).

 In this 
 note, we consider the \emph{consequences} of 
$\bbA\MMod_\cat C$
 being triangulated.
We prove 
that if this is the case then
 \emph{any} triangulated realization of~$\bbA$ is essentially monadic, \ie monadic after applying two necessary operations: a Verdier quotient and an idempotent completion.
In a slogan:

\begin{quotation}
	Triangulated monads which have triangulated Eilenberg--Moore adjunctions are intrinsically monadic.
\end{quotation}
The proof of this amusing fact will be given in Theorem~\ref{thm:essential_monadicity} 
below 
(see also Corollary~\ref{cor:main}).
An application to categories of module spectra will be discussed at the end (see Corollary~\ref{cor:equiv_ring}).

\begin{Ter} \label{ter:monads}
We recall some basic facts about monads from \cite[Chap.~VI]{MacLane98}, mostly to fix 
notation.
Every adjunction $F\colon \cat C\rightleftarrows \cat D:\!G$ with unit $\eta\colon \id_\cat C\to GF $ and counit $\varepsilon\colon FG\to \id_\cat D$ defines a monad $\bbA$ on $\cat C$ consisting of the endofunctor $\bbA:=GF\colon \cat C\to \cat C$ equipped with the multiplication map $G\varepsilon F\colon \bbA^2\to \bbA$ and unit map $\eta\colon \id_\cat C\to \bbA$. We say the adjunction $F\dashv G$ \emph{realizes} the monad~$\bbA$.
Given any monad $\bbA$ on~$\cat C$, there always exist an initial and a final adjunction realizing~$\bbA$:
\begin{equation} \label{eq:big_diag}
	\begin{gathered}
\xymatrix{
&& \cat C 
 \ar@<-2pt>[d]_-F 
 \ar@<-2pt>[drr]_-{F_\bbA} 
 \ar@<-2pt>[dll]_-{F_\bbA} && \\
\bbA \FFree_{\cat C} 
 \ar@<-2pt>[urr]_-{U_\bbA} 
 \ar[rr]_{\exists !\, K} 
 &&
 \cat D
  \ar@<-2pt>[u]_G 
  \ar[rr]_-{\exists!\, E}   &&
 \bbA\MMod_{\cat C}. 
 \ar@<-2pt>[ull]_-{U_\bbA}
}
	\end{gathered}
\end{equation}
The final one is provided by the \emph{Eilenberg--Moore} category $\bbA\MMod_\cat C$, whose objects are $\bbA$-modules $(x, \rho\colon \bbA x\to x)$ in $\cat C$, together with the forgetful functor $U_\bbA\colon {(x,\rho)\mapsto x}$ and its left adjoint \emph{free-module} functor~$F_\bbA$. The full image of $F_\bbA$, together with the restricted adjunction, provides the initial realization~$\bbA\FFree_\cat C$ (often called the \emph{Kleisli category}). For any 
adjunction
$F\dashv G$
realizing the monad~$\bbA$, the fully faithful inclusion $\bbA \FFree_\cat C\to \bbA \MMod_\cat C$ uniquely factors as a composite $E\circ K$ of two comparison functors 
satisfying
$KF_\bbA= F, GK=U_\bbA$ and $EF = F_\bbA,U_\bbA E=G$. The functor $K$ is always automatically fully faithful.
Finally, an adjunction $F\dashv G$ is \emph{monadic} if the associated Eilenberg--Moore comparison $E$ is an equivalence.
\end{Ter}

\begin{Rem}
	If $\cat C$ is a triangulated category then we can also consider \emph{triangulated realizations} of $\bbA$,
	\ie realizations by an
	adjunction $F\colon \cat C \adjto \cat D:\!G$ of exact functors between triangulated categories.
	(Of course, $\bbA\colon \cat C \to \cat C$ must be exact for 
	a triangulated realization to exist.)
Note that if $\bbA\MMod_\cat C$ is triangulated such that $U_\bbA$ is exact, then
the free module functor $F_\bbA\colon \cat C\to \bbA\MMod_\cat C$ is also automatically exact, hence the adjunction $F_\bbA\dashv U_\bbA$ is a triangulated realization of~$\bbA$. 
However, $\bbA\MMod_\cat C$ need not admit such a triangulation in general (see \eg \cite[Ex.\,5.5]{Lagkas-Nikolos18}).
\end{Rem}

\begin{Rem} \label{rem:necessary}
As $\cat C$ is assumed to be idempotent complete, one easily checks that the Eilenberg--Moore category $\bbA\MMod_\cat C$ is also idempotent complete. Moreover $U_\bbA$ is faithful, so in particular it detects the vanishing of objects; since $U_\bbA$ is an exact functor, the latter is equivalent to being \emph{conservative}, \ie reflecting isomorphisms.
\end{Rem}

\begin{Rem} \label{rem:adjustable}
If
$F:\cat C\rightleftarrows \cat D: G$ is any triangulated realization of~$\bbA$, then we can always canonically modify it to a triangulated adjunction $\widetilde F: \cat C\rightleftarrows \widetilde{\cat D} : \widetilde G$ where, as with the Eilenberg--Moore adjunction, the target category $\widetilde{\cat D}$ is idempotent complete and the right adjoint $\widetilde G$ is conservative.
Indeed, construct the Verdier quotient $\cat D/ \Ker G$,
embed it into its idempotent completion $(\cat D/\Ker G)^\natural$ (see~\cite{BalmerSchlichting01}), and let $\widetilde F$ be the composite $\cat D \to \cat D/\Ker G \to (\cat D/\Ker G)^\natural =: \widetilde{\cat D}$.
Since $\cat C$ is idempotent complete, $G$ extends to a functor $\widetilde G\colon \widetilde{\cat D}\to \cat C$ right adjoint to $\widetilde F$ and one easily verifies that the adjunctions $\widetilde F\dashv \widetilde G $ and $ F\dashv G$ realize the same monad.
\end{Rem}

	Our aim is to study all possible triangulated realizations of $\bbA$ under the hypothesis that the Eilenberg--Moore adjunction is triangulated.
By Remarks \ref{rem:necessary} and \ref{rem:adjustable}, this problem reduces to the case where the target category is idempotent complete and the right adjoint is conservative.

Surprisingly, such an adjunction is necessarily monadic:

\begin{Thm} \label{thm:essential_monadicity}
Let $\cat C$ be an idempotent complete triangulated category equipped with a monad $\bbA$ such that 
$\bbA \MMod_\cat C$ is compatibly triangulated, \ie such that the free-forgetful adjunction $F_\bbA \dashv U_\bbA$ is a triangulated realization of~$\bbA=U_\bbA F_\bbA$.  Let $F: \cat C \rightleftarrows \cat D: G$ be any triangulated realization of the same monad $\bbA = GF$. If $\cat D$ is idempotent complete and $G$ is conservative, then $F\dashv G$ is monadic, \ie the comparison functor $\cat D \stackrel{\sim}{\to} \bbA \MMod_\cat C$ is an equivalence.
\end{Thm}

\begin{proof}
Keep in mind \eqref{eq:big_diag} throughout and consult \cite[Chap.~VI]{MacLane98} if necessary.
Let $d\in \cat D$ be an arbitrary object. By definition, the $\bbA$-module $Ed$ consists of the object $Gd\in \cat C$ equipped with the action 
\[ G\varepsilon_d\colon \bbA (Gd)= G FG d \longrightarrow Gd \]
where $\varepsilon$ denotes the counit of the adjunction $F\dashv G$.
As for any module, its action map can also be seen as a map in $\bbA\MMod_\cat C$
\[ G\varepsilon_d \colon F_\bbA U_\bbA Ed= (GFG d, G\varepsilon_{FGd}) \longrightarrow (Gd, G\varepsilon_d) = Ed\] 
providing the counit at the object $Ed$ for the Eilenberg--Moore adjunction.
By hypothesis, the right adjoint $U_\bbA\colon \bbA\MMod_\cat C\to \cat C$ of $F_\bbA$ is a faithful exact functor between triangulated categories, hence the counit $G\varepsilon_d$ admits a section $\sigma \colon Gd\to GFGd = F_\bbA U_\bbA Ed$ in $\bbA\MMod_\cat C$ (see e.g.\ \cite[Lemma~4.2]{BalmerDellAmbrogioSanders16}).
Thus we have in $\bbA\MMod_\cat C$ the split idempotent $p^2=p:=\sigma \circ G\varepsilon_d$ on the object $F_\bbA U_\bbA Ed$ with image~$Ed$:
\begin{equation*} 
\xymatrix @R=\Rdef{
F_\bbA U_\bbA Ed \ar[dr]_{G\varepsilon_d} \ar[rr]^-p && F_\bbA U_\bbA Ed\\
& Ed \ar[ur]_\sigma &
}
\end{equation*}
Since the composite functor $EK$ is fully faithful and the free module $F_\bbA U_\bbA Ed$ belongs to its image, we must have $p= EKq$ for an idempotent~$q$ in $\bbA\FFree_\cat C$, hence $p=Er$ for the idempotent $r:= Kq$ in~$\cat D$ on the object $FGd$. As $\cat D$ is idempotent complete by hypothesis, $r$ must split:
\begin{equation} \label{eq:r}
	\begin{gathered}
\xymatrix @R=\Rdef{
&\exists\, d' \ar[dr]^{\beta} & \\
FGd \ar[ur]^\alpha \ar[rr]_-r && FGd.
}
\end{gathered}
\end{equation}
Applying $E$ to~\eqref{eq:r},
we see that the idempotent $p$ of $\bbA\MMod_\cat C$ splits in two ways:
\[
\xymatrix @R=\Rdef{
& Ed' \ar[dr]^{E\beta} & \\
E FGd \ar[ur]^{E\alpha} \ar[dr]_{G\varepsilon_d} \ar[rr]^-{p\, = \, Er} && EFGd\\
& Ed \ar[ur]_\sigma &
}
\]
Applying the functor $U_\bbA$ (\ie forgetting actions), this yields in $\cat C$ the two splittings
\[
\xymatrix @R=\Rdef{
& Gd' \ar[dr]^{G\beta} & \\
GFGd \ar[ur]^{G\alpha} \ar[dr]_{G\varepsilon_d} \ar[rr]^-{p} && GFGd\\
& Gd \ar[ur]_\sigma &
}
\]
of the idempotent $p$ on $GFGd$. 
It follows that the composite $G(\varepsilon_d) \circ G(\beta) = G(\varepsilon_d\circ \beta)$ 
is an isomorphism $Gd' \cong Gd$
of the two images.
Since $G$ is assumed to be conservative, this implies that $\varepsilon_d \circ \beta$ is already an isomorphism $d'\cong d$ in~$\cat D$.
We conclude from \eqref{eq:r} that $d$ is a retract of an object $FGd = K(F_\bbA Gd)$ in the image of the functor~$K$.

As $d\in \cat D$ was arbitrary, we have proved that the fully faithful functor $K$ is surjective up to direct summands. Consider now the idempotent completions
\begin{equation*} 
\xymatrix @R=\Rdef{
\bbA \FFree_{\cat C} 
 \ar[rr]_{\exists !\, K} 
 \ar@{^{(}->}[d] &&
 \cat D
 \ar@{^{(}->}[d]^\simeq
  \ar[rr]_-{\exists!\, E}   &&
 \bbA\MMod_{\cat C} 
 \ar@{^{(}->}[d]^\simeq
 \\
(\bbA \FFree_{\cat C})^\natural 
 \ar[rr]_{K^\natural} &&
 \cat D^\natural
  \ar[rr]_-{E^\natural}   &&
 (\bbA\MMod_{\cat C})^\natural
}
\end{equation*}
together with the induced functors $K^\natural$ and~$E^\natural$.
The two rightmost canonical inclusions are equivalences, since $\cat D$ and $\bbA\MMod_\cat C$ are idempotent complete. What we have just proved amounts to $K^\natural$ being an equivalence too, \ie the Kleisi comparison functor induces an equivalence 
\[
K^\natural\colon (\bbA\FFree_\cat C)^\natural \stackrel{\sim}{\longrightarrow} \cat D^\natural \cong \cat D
\]
after idempotent completion. 

We can now easily see that  $E$ is an equivalence by chasing the above diagram: as $(EK)^\natural=E^\natural K^\natural $ is fully faithful and $K^\natural$ is an equivalence, $E^\natural$ is fully faithful, hence so is~$E$. 
As already observed, the faithfulness of $U_\bbA$ implies that $F_\bbA = EF$ is surjective up to summands, so $E$ must be too. But $E$ is fully faithful and its domain $\cat D$ is idempotent complete, hence it must be essentially surjective.
%
\end{proof}

\begin{Rem} \label{rem:exactness}
	If the triangulated structure on $\bbA\MMod_\cat C$ is \emph{inherited} from $\cat C$ in the sense that the forgetful functor $U_\bbA$ \emph{creates} the triangulation of $\bbA\MMod_\cat C$
 (\ie a triangle in $\bbA\MMod_\cat C$ is exact if and only if its image under $U_\bbA$ is exact in~$\cat C$) then the equivalence $E \colon \cat D \stackrel{\sim}{\to} \bbA \MMod_\cat C$
	of Theorem~\ref{thm:essential_monadicity}
	is an exact equivalence of triangulated categories.
	This is because $U_\bbA E=G$ is exact by hypothesis and $U_\bbA$ would then reflect exact triangles.
Similarly, we can conclude that $E$ is an exact equivalence by instead assuming that $G$, rather than $U_\bbA$, reflects exact triangles.
\end{Rem}

\begin{Cor}\label{cor:main}
	Let $\bbA$ be a monad on an idempotent complete triangulated category~$\cat C$.  If the Eilenberg--Moore adjunction $\cat C \rightleftarrows \bbA\MMod_\cat C$ is triangulated, then any triangulated realization $F\colon \cat C\rightleftarrows \cat D:\! G$ of $\bbA$ induces canonical equivalences:
\[
(\bbA\FFree_\cat C)^\natural 
\stackrel{\sim}{\longrightarrow}
(\cat D/ \Ker G)^\natural
\stackrel{\sim}{\longrightarrow}
\bbA\MMod_\cat C \,.
\]
\end{Cor}

\begin{Exa}
	Let $\cat C:=\SH$ denote the stable homotopy category of spectra.
	If $A$ is a highly structured ring spectrum ($S$-algebra, $\mathrm A_\infty$-ring spectrum, brave new ring,\,\ldots),
	then we may consider its \emph{derived category} $\Der(A)$, defined to be the homotopy category of highly structured $A$-modules; see \eg \cite{EKMM97}. 
	The unit map $f\colon S\to A$ induces a triangulated 
	adjunction $f^*= A\wedge -\colon \SH = \Der(S) \rightleftarrows \Der(A):\! \Hom_A(A_S, -)=f_*$. 
On the other hand, by forgetting structure, $A$ is also a monoid in $\SH$ and therefore we may consider modules over it in $\SH$. The resulting category $A\MMod_{\SH}$ of \emph{na\"ive} or \emph{homotopy} $A$-modules is nothing else but the Eilenberg--Moore category for the monad associated with the adjunction $f^*\dashv f_*$.  Thus we obtain a comparison functor $E$ as in~\eqref{ter:monads}, which can be thought of as forgetting the higher structure of an $A$-module:
\[
\xymatrix{
\SH \ar@<-2pt>[d]_-{f^*} \ar@<-2pt>[drr]_-{F_A} && \\
\Der(A) \ar@<-2pt>[u]_{f_*} \ar[rr]_-{\exists!\, E} && A\MMod_{\SH} \ar@<-2pt>[ull]_-{U_A}
}
\]
Note that both triangulated categories $\SH$ and $\Der(A)$ are idempotent complete; \eg because they admit infinite coproducts. The right adjoint $f_*$ is conservative by construction; this is equivalent to $A=f^*S \in \Der(A)$ weakly generating $\Der(A)$. 
Moreover $f_*$ creates the triangulation of $\Der(A)$, again by construction.
Hence Theorem~\ref{thm:essential_monadicity} and Remark~\ref{rem:exactness} immediately imply the following result.
\end{Exa}

\begin{Cor} \label{cor:equiv_ring}
Let $A$ be any highly structured ring spectrum. 
Then the category $A \MMod_{\SH}$ of na\"ive $A$-modules is triangulated, in a way that makes the forgetful functor to $\SH$ exact, if and only if the canonical comparison functor $\Der(A)\to A \MMod_{\SH}$ is a (necessarily exact) equivalence.
\qed
\end{Cor}

\begin{Exa}
	\label{exa:EML}
The comparison between strict and na\"ive modules was studied by
Guti\'errez~\cite{Gutierrez05} in the special case where  $A=\EM R$ is the
Eilenberg--Mac~Lane spectrum of an ordinary associative and unital ring~$R$.  He
showed that the comparison map is an equivalence if $R$ is a field or a subring
of~$\bbQ$, for instance the ring of integers~$\bbZ$.  
\end{Exa}

\subsection*{Acknowledgements:}
We would like to thank Paul Balmer and Javier Guti\'er\-rez for useful discussions on these topics.

\bibliographystyle{alpha}%
\bibliography{TG-articles}

\end{document}